\documentclass[a4paper,fleqn]{cas-sc}

\usepackage[numbers]{natbib}
\usepackage{amsthm}
\usepackage{cleveref}  
\usepackage{enumitem}  
\usepackage{longtable}   
\newtheorem{definition}{Definition}[section]
\newtheorem{conjecture}{Conjecture}  
\newtheorem{theorem}{Theorem}[section]
\newtheorem{proposition}{Proposition}[section]
\newtheorem{corollary}{Corollary}[section]
\usepackage{caption}       
\usepackage{float}         

\def\tsc#1{\csdef{#1}{\textsc{\lowercase{#1}}\xspace}}
\tsc{WGM}
\tsc{QE}

\begin{document}
\let\WriteBookmarks\relax
\def\floatpagepagefraction{1}
\def\textpagefraction{.001}

\shorttitle{}    

\shortauthors{Xu Wang and Jiayi Zhu}  

\title [mode = title]{On Base , Normal and Near-normal Sequences}  



%

\author[]{Xu Wang}

\cormark[1]


\ead{003796@nuist.edu.cn}



\affiliation[a]{organization={School of Electronics and Information Engineering, Nanjing University of Information Science and Technology},
            addressline={219 Ningliu Road}, 
            city={Nanjing},
            postcode={210044}, 
            state={Jiangsu Province},
            country={P.R.C.}}

\author[]{Jiayi Zhu}

\cortext[1]{Corresponding author}



\begin{abstract}
The base sequences BS(n+1,n) are four sequences of $\pm1$ and lengths n+1,n+1,n,n with zero auto correlation. The base sequence conjecture states that BS(n+1,n) exists for all positive integers and has been verified for $n\le40$. We present our algorithm and give construction of BS(n+1,n) for $n=41,42,43$.\\
The Normal sequences NS (n) and the Near-normal sequences NNS (n) are subclasses of BS(n+1,n). Yang conjecture asserts that there is a NNS(n) for each even integer n and has been verified for $n\le40$. We found that there is no NNS(n) for n=42 and 44 by exhaustive search, which gives the first counter case of Yang conjecture. We also show that there is no NS(n) for n=41,42,43,44,45 by exhaustive search and proves that no NS(n) exist for $n=8k-2,k \in Z_+$.
\end{abstract}


\begin{highlights}
\item A new algorithm to search base sequence is proposed , example of base sequence of lengths 41,42,43 is given for the first time. 
\item Near-normal sequences is conjectured to exist for all even integers(Yang conjecture), we verified that 42 and 44 are the smallest 2 counterexamples of Yang conjecture. 
\item We prove that normal sequence do not exist for $n=8k-2,k \in Z_+$. We also verified that normal sequence do not exist for n=41-45 by computer search. 
\end{highlights}


\maketitle 

\section{Introduction}
\label{Section:intro}
The base sequences are four sequences of $\pm1$ and lengths m,m,n,n with zero auto correlation, which was first introduced by Turyn\cite{turyn1974} for the construction of Hadamard matrix. The base sequence conjecture(BSC)\cite{colbourn2007}(P319) asserts that BS(n+1,n) exists for all $n \in Z_+$. As pointed out by Ðoković\cite{dokovic1998}, the BSC is a stronger conjecture than the famous Hadamard conjecture. Currently BSC is verified for all $n \le 40$ and all Golay numbers ($n = 2^a10^b26^c$ with $a,b,c \ge 0$)\cite{dokovic2010small}.

For a given sequence A=($a_{1}$,$a_{2}$,...,$a_{n}$),the non-periodic autocorrelation function $N_{A}$ is defined as
\begin{equation}
N_{A} (i)=\sum_{j\in Z} a_{j}a_{i+j} , \enspace where \enspace a_{k}=0,if \enspace k<1\enspace or\enspace k>n
\label{Equation:11}
\end{equation}
\begin{definition}
The base sequences,denoted BS(m,n), are four sequences A,B,C,D of $\pm1$ and lengths m,m,n,n with zero auto-correlation, which can be expressed as
\label{Definition:11}
\begin{equation}
N_{A}(i)+N_{B}(i)+N_{C}(i)+N_{D}(i)=\left\{
\begin{array}{ccl}
2m+2n &,  & {i=0}\\
0     &,  & {i=1,...,n-1}
\end{array}\right.
\label{Equation:12}
\end{equation}
\end{definition}
\begin{definition}
(A,B,C,D) $\in$ BS(n+1,n) are normal sequences,denoted NS(n), if
\begin{equation}
\left\{
\begin{array}{ccl}
b_i=a_i    &,  & 1 \le i \le n\\
a_{n+1}=1  &,  & b_{n+1}=-1
\end{array}\right.
\label{Equation:13}
\end{equation}
\label{Definition:12}
\end{definition}

\begin{definition}
(A,B,C,D) $\in$ BS(n+1,n) is a near-normal sequence,denoted NNS(n), if n is even,and
\begin{equation}
\left\{
\begin{array}{ccl}
b_i=(-1)^{i-1}a_i   &,  & 1 \le i \le n\\
a_{n+1}=1           &,  & b_{n+1}=-1
\end{array}\right.
\label{Equation:14}
\end{equation}
\label{Definition:13}
\end{definition}
Normal Sequences and Near-normal Sequences can be viewed as sub classes of BS(n+1,n), they are introduced by Yang for the construction of T-sequences \cite{yang1989}. If there exist Normal Sequences and Near-normal Sequences of length t and Base Sequences BS(m,n), then a T-sequences of length $(m+n)(2t+1)$ can be constructed, and we call $y=2t+1$ the Yang number.

\begin{theorem} \cite{cohen1988,dokovic2010hadamard}
If Base sequences BS(m,n),Yang number y,Williamson-type matrices WT(w) and Baumert-Hall-Welch arrays BHW (4h) exist, then Hadamard matrix HM (4n) exists for $n=yh(m+n)w$.
\label{Theorem:14}
\end{theorem}

Base, Normal and Near-normal Sequences are important for the construction of the Hadamard matrix, the method is described in \Cref{Theorem:14}. For example, BS(71,36) is used for the construction of Hadamard matrix of order 428 \cite{kharaghani2005}, it is also believed that this may be the most promising approach to construct H(668), the smallest unknown case for Hadamard matrix. Therefore, the study of Base, Normal and Near-normal Sequences is of interest and has been studied by many previous works.\Cref{Proposition:15} lists the state of the art for Base, Normal and Near-normal Sequences.

\begin{proposition}
\leavevmode
\begin{enumerate}
  \item $\mathrm{BS}(n+1,n)$ exists for all $n \le 40$ and for all Golay numbers ($n = 2^a10^b26^c$ with $a,b,c \ge 0$). The first unknown case is $\mathrm{BS}(42,41)$~ \cite{dokovic2010small}.
  \item $\mathrm{NNS}(n)$ exists for all even integers $n \le 40$. The first unknown case is $\mathrm{NNS}(42)$~ \cite{dokovic2010hadamard}.
  \item $\mathrm{NS}(n)$ exist for all the Golay numbers and $n=3,5,7,9,11,12,13,25,29$, $\mathrm{NS}(n)$ do not exist for all the other $n\le40$ The first unknown case is $\mathrm{NS}(41)$. \cite{dokovic2011}.  
  \item All odd numbers $n \le 81$ are the Yang numbers except $35,43,47,55,63,67,71,75,79$.
\end{enumerate}
\label{Proposition:15}
\end{proposition}

The paper is organized as follows: In \Cref{Section:facts} we recall properties of BS(n+1,n) that we use for our search. In \Cref{Section:search} we describe the outline of our search algorithm for BS(n+1,n) and give the first construction of BS(n+1,n) for $n=41,42,43$. In \Cref{Section:near} we discuss the exhaustive search method for Near-normal sequences based on the algorithm in \Cref{Section:search}, our result shows that NNS(n) do not exist for $n=42,44$ , this gives the first counter case of the Yang conjecture. In \Cref{Section:normal} we discuss the exhaustive searching method for Normal sequences. We prove that NS(n) do not exist for $n=8k-2,k \in Z_+$. We also perform exhaustive searching for NS(41)-NS(45) and it turns out that none of these sequences exists, therefore 83,85,87,89,91 are not Yang numbers. 

The base sequence conjecture holds for $n \le 43$ and remains open, the Yang conjecture is not true for n=42 and 44, and the existence of BS(n+1,n) for $n>43$, NNS(n) for even integer $n>44$ and NS(n) for $n>46$ is still open.

\section{Facts on Base Sequences(n+1,n)}
\label{Section:facts}
In this section, we mention all the properties we use in our search algorithm about the Base Sequences(n+1,n).

Firstly we introduce the isomorphic transformations of BS(n+1,n). 
For a given sequence A=($a_{1}$,$a_{2}$,...,$a_{n}$),the negated sequence $-$A, reversed sequence A$'$ and alternated sequence $A^*$ are defined by
\begin{equation}
\begin{array}{rcl}
-A&=&-a_1,-a_2,...,-a_n\\
A'&=&a_n,a_{n-1},...,a_1\\
A^*&=&a_1,-a_2,a_3,-a_4,...,(-1)^{n-1}a_n
\end{array}
\label{Equation:21}
\end{equation}

There are five isomorphic transformations of (A;B;C;D)$\in$ BS(n+1,n)\cite{dokovic2010classification}:

\begin{enumerate}[label=(\roman*), leftmargin=*, itemsep=0.5ex]
  \item Negate one or more sequences of $A, B, C, D$.
  \item Reverse one or more sequences of $A, B, C, D$.
  \item Interchange the sequences $A, B$ or $C, D$.
  \item Alternate all four sequences of $A, B, C, D$.
  \item Replace 
    $\begin{bmatrix} c_i & c_{n+1-i} \\ d_i & d_{n+1-i} \end{bmatrix}$ 
    that are 
    $\begin{bmatrix} 1 & -1 \\ -1 & 1 \end{bmatrix}$ 
    and 
    $\begin{bmatrix} -1 & 1 \\ 1 & -1 \end{bmatrix}$ 
    with 
    $\begin{bmatrix} -1 & 1 \\ 1 & -1 \end{bmatrix}$ 
    and 
    $\begin{bmatrix} 1 & -1 \\ -1 & 1 \end{bmatrix}$, respectively.
\end{enumerate}

Two base sequences are equivalent if one can be transformed to the other by applying one or more transformations listed above.

Then we introduce the properties related with the sum of each sequence.

Let $A(z)=a_1+a_2z+...+a_nz^{n-1}$ , we refer to the Laurent polynomial $N(A)=A(z)A(z^{-1})$ as the norm of A,then
\begin{equation}
N(A)=A(z)A(z^{-1})=\sum_{i\in Z}N_{A}(i)(z^{i}+z^{-i}) , z\ne0
\label{Equation:22}
\end{equation}

With \Cref{Definition:11}, we have for BS(m,n)
\begin{equation}
N(A)+N(B)+N(C)+N(D)=2(m+n)
\label{Equation:23}
\end{equation}

\begin{theorem}
\label{Theorem:21}
Let A,B,C,D be BS(n+1,n),and denote the sum of sequences A;B;C;D and $A^*;B^*;C^*;D^*$ as a,b,c,d and $a^*,b^*,c^*,d^*$.
\begin{equation}
\label{Equation:24}
\begin{aligned}
&a^2+b^2+c^2+d^2=4n+2\\
&a^{*2}+b^{*2}+c^{*2}+d^{*2}=4n+2 \\
&a\equiv b\equiv a^*\equiv b^*\equiv n+1\enspace(mod\enspace 2)\\
&c\equiv d\equiv c^*\equiv d^*\equiv n\enspace(mod\enspace 2)
\end{aligned}
\end{equation}
and the modular 4 properties: 
\begin{equation}
\label{Equation:25}
\begin{aligned}
&c\equiv d,c^*\equiv d^*&(mod\enspace4)&\enspace when \enspace n\enspace even\\
&a\equiv b+2,a^*\equiv b^*+2&(mod\enspace4)&\enspace when\enspace n\enspace odd\\
\end{aligned}
\end{equation}

\begin{equation}
\label{Equation:26}
\begin{aligned}
&a\equiv a^*,b\equiv b^*,c\equiv c^*,d\equiv d^*,&when\enspace  n\equiv 0 (mod\enspace4)\\
&a\equiv a^*+2,b\equiv b^*+2,c\equiv c^*,d\equiv d^*,&when\enspace  n\equiv 1 (mod\enspace4)\\
&a\equiv a^*+2,b\equiv b^*+2,c\equiv c^*+2,d\equiv d^*+2,&when\enspace  n\equiv 2 (mod\enspace4)\\
&a\equiv a^*,b\equiv b^*,c\equiv c^*+2,d\equiv d^*+2,&when\enspace  n\equiv 3 (mod\enspace4)
\end{aligned}
\end{equation}

\end{theorem}

\begin{proof}
Let z=1 and -1 in \Cref{Equation:23}, we will have \Cref{Equation:24}. And all the elements in sequences are $\pm1$, therefore we have the modular 2 equation.\\
The proof for \Cref{Equation:25} use the fact that if an integer a is odd, then 
$a^2\equiv1 \enspace(mod\enspace8)$\\
When n is even, a,b are odd and c,d are even.\\
With $a^2+b^2+c^2+d^2=4n+2\equiv2 \enspace(mod\enspace8)$\\
We have $c^2+d^2\equiv0 \enspace(mod\enspace8)$, thus $c\equiv d\enspace (mod\enspace4)$\\
Similarly, when n is odd, a,b are even and c,d are odd.\\
With $a^2+b^2+c^2+d^2=4n+2\equiv6 \enspace(mod\enspace8)$\\
We have $a^2+b^2\equiv4 \enspace(mod\enspace8)$, thus $a\equiv b+2\enspace (mod\enspace4)$\\
In the proof for \Cref{Equation:26}, we denote $a_o$ and $a_e$ for the sum of odd and even terms in sequence A. $a=a_o+a_e,a^*=a_o-a_e$, when $a_e$ is even, $a\equiv a^*$, and when $a_e$ is odd, $a\equiv a^*+2(mod\enspace4)$
For each case in \Cref{Equation:26}, we can know whether $a_e,b_e,c_e,d_e$ is even or odd, and then we prove \Cref{Equation:26}.
\end{proof}
\begin{corollary}
\label{Corollary:21}
For NNS(n), in addition to \Cref{Theorem:21}, we have:
\begin{equation}
\label{Equation:41}
a=b^*+2,b=a^*-2
\end{equation}
For NS(n), in addition to \Cref{Theorem:21}, we have:
\begin{equation}
\label{Equation:51}
\left\{
\begin{array}{lll}
a=b+2,&a^*=b^*-2,&n\enspace odd\\
a=b+2,&a^*=b^*+2,&n\enspace even
\end{array}\right.
\end{equation}
\end{corollary}
\begin{proof}
For NNS(n), n must be even. With \Cref{Definition:13}, we have $a_o=b_o+2,a_e=-b_e$, then\\
$a=a_o+a_e=b_o+2-b_e+2=b^*+2$, 
$b=b_o+b_e=a_o-2-a_e+2=a^*-2$\\
For NS(n), With \Cref{Definition:12}, when n is odd, we have $a_o=b_o,a_e=b_e+2$, then\\
$a=a_o+a_e=b_o+b_e+2=b+2,  a^*=a_o-a_e=b_o-b_e-2=b^*-2$\\
when n is even, we have $a_o=b_o+2,a_e=b_e$, then\\
$a=a_o+a_e=b_o+b_e+2=b+2,  a^*=a_o-a_e=b_o+2-b_e=b^*+2$\\
\end{proof}
\Cref{Theorem:21} comes from \cite{dokovic1998} and is useful to narrow down the search space. Our search for Base(n+1,n) always starts from one of the $a,b,c,d,a^*,b^*,c^*,d^*$ sets obtained from \Cref{Theorem:21}. Our search for NS(n) and NNS(n) starts by finding all possible $a,b,c,d,a^*,b^*,c^*,d^*$ sets with \Cref{Theorem:21} and \Cref{Corollary:21}. We also prove that we cannot find such sets for NS(n) when $n=8k-2,k \in Z_+ $ in \Cref{Section:normal}.

\begin{theorem}\cite{koukouvinos1990,turyn1974}
\label{Theorem:22}
Let A,B,C,D be BS(n+1,n),then
\begin{equation}
\label{Equation:27}
\begin{aligned}
&a_{i}+b_{i}+a_{n+2-i}+b_{n+2-i}=\left\{
\begin{array}{rcl}
2\enspace mod\enspace 4, &  & {i=1}\\
0\enspace mod\enspace 4, &  & {i=2,...,[(n+1)/2]}
\end{array}\right.
\\
&c_{i}+d_{i}+c_{n+1-i}+d_{n+1-i}=0\enspace mod\enspace4,  {i=2,...,[n/2]}
\end{aligned}
\end{equation}
\end{theorem}

\Cref{Theorem:22} shows that for each $\begin{bmatrix} a_i & a_{n+2-i}  \\ b_i & b_{n+2-i}  \end{bmatrix}$ and $\begin{bmatrix} c_i & c_{n+1-i}  \\ d_i & d_{n+1-i}  \end{bmatrix}$, there is only 8 possible case. This property is widely used in the previous works.

\begin{theorem}\cite{koukouvinos1990}
\label{Theorem:23}
Let A,B,C,D be BS(n+1,n) and given m $\in$ {2,3,...,n+1}, denote\\
\begin{equation}
\label{Equation:28}
\begin{aligned}
&k_{im} &= \sum_{j\equiv i \pmod{m}} a_j,&
&r_{im} &= \sum_{j\equiv i \pmod{m}} b_j,\\
&p_{im} &= \sum_{j\equiv i \pmod{m}} c_j, &
&q_{im} &= \sum_{j\equiv i \pmod{m}} d_j.
\end{aligned}
\end{equation}

\begin{equation}
\label{Equation:29}
\begin{aligned}
&N_{K}(s) &= \sum_{i=1}^{m-s} k_{im}k_{i+s,m}, &
&N_{R}(s) &= \sum_{i=1}^{m-s} r_{im}r_{i+s,m}, \\
&N_{P}(s) &= \sum_{i=1}^{m-s} p_{im}p_{i+s,m}, &
&N_{Q}(s) &= \sum_{i=1}^{m-s} q_{im}q_{i+s,m}.
\end{aligned}
\end{equation}

then for the given m $\in$ {2,3,...,[(n+1)/2]},we have
\begin{equation}
\label{Equation:210}
\begin{aligned}
&N_{K}(0)+N_{R}(0)+N_{P}(0)+N_{Q}(0)=k_{1m}^2+\cdots + k_{mm}^2 + r_{1m}^2+\cdots + r_{mm}^2\\
&+p_{1m}^2+\cdots + p_{mm}^2 + q_{1m}^2+\cdots + q_{mm}^2=4n + 2\\
&N_{K}(s)+N_{R}(s)+N_{P}(s)+N_{Q}(s)+N_{K}(m - s)+N_{R}(m - s)\\
&+N_{P}(m - s)+N_{Q}(m - s)=0, \quad s = 1, \ldots, [m/2].
\end{aligned}
\end{equation}
and for the $k_{im},r_{im},p_{im},q_{im}$, we can also get the following range and module 4 property: 
\begin{equation}
\label{Equation:211}
\begin{aligned}
&\left | k_{im} \right | \le \left [ \frac{n+1-i}{m}  \right ] +1,\left | r_{im} \right | \le \left [ \frac{n+1-i}{m}  \right ] +1,i=1,...,m\\
&\left | p_{im} \right | \le \left [ \frac{n-i}{m}  \right ] +1,\left | q_{im} \right | \le \left [ \frac{n-i}{m}  \right ] +1,i=1,...,m\\
&k_{im}\equiv r_{im}\equiv\left [ \frac{n+1-i}{m}  \right ] +1,p_{im}\equiv q_{im}\equiv\left [ \frac{n-i}{m}  \right ] +1,\pmod{2}
\end{aligned}
\end{equation}

\begin{equation}
\label{Equation:212}
\begin{aligned}
&k_{1m}+r_{1m}+k_{n + 1,m}+r_{n + 1,m} \equiv 
\begin{cases}
2\pmod{4} & \text{if } n\not\equiv0\pmod{m}, \\
0\pmod{4} & \text{if } n\equiv0\pmod{m},
\end{cases}\\
&k_{jm}+r_{jm}+k_{n + 2 - j,m}+r_{n + 2 - j,m} \equiv 0\pmod{4}, \quad j = 2, \ldots, m, \\
&p_{jm}+q_{jm}+p_{n + 1 - j,m}+q_{n + 1 - j,m} \equiv 0\pmod{4}, \quad j = 1, \ldots, m.
\end{aligned}
\end{equation}
\end{theorem}

For the proof of \Cref{Theorem:23} ,please refer to \cite{koukouvinos1990}, where this property is used to search the Base sequence, the main idea is to find values of $k_{1m},...,k_{mm},r_{1m},...,r_{mm},p_{1m},...,p_{mm},q_{1m},...,q_{mm}$ defined in \Cref{Equation:28} and satisfies \Cref{Equation:210,Equation:211,Equation:212}, then continue to find values of $k_{1,2m},...,k_{2m,2m},r_{1,2m},...,r_{2m,2m},p_{1,2m},...,p_{2m,2m},q_{1,2m},...,q_{2m,2m}$, and repeat this process until $m\ge n+1$, then check if the sequence obtained is Base sequence. One problem of this method is with the increase of the order of Base sequence and the m chosen, the number of combinations of $k_{1m},...,k_{mm},r_{1m},...,r_{mm},p_{1m},...,p_{mm},q_{1m},...,q_{mm}$ will increase exponentially and become impossible to process. In our algorithm we only find values of $k_{1m},...,k_{mm},r_{1m},...,r_{mm}$ or $p_{1m},...,p_{mm},q_{1m},...,q_{mm}$ for a small number m. The benefit is that the overall combinations is relatively small, and we can use the result to to find the candidate A,B or C,D sequences for BS(n+1,n) first. We can further filter the sequences obtained by the following property.

For a given sequence A, the associate Hall polynomial $h_{A}$ if defined by $h_{A}(t)=\sum_{i=0}^{n-1} a_{i} t_{i} $

Let the real function $f_{A}$ be defined as $f_{A}=\left | h_{A} (e^{i\theta} ) \right | $, then we have 
\begin{equation}
\label{Equation:215}
f_{A}(\theta)=N_{A}(0)+2\sum_{i=1}^{n-1} N_{A}(j)cosj\theta
\end{equation}

\begin{theorem}
\label{Theorem:24}
Let A,B,C,D be BS(n+1,n) and $f_{A}$ be defined as \Cref{Equation:215},then
\begin{equation}
\label{Equation:216}
(f_{A}+f_{B}+f_{C}+f_{D})(\theta)=4n+2,\enspace for\enspace all\enspace \theta
\end{equation}
\end{theorem}

\Cref{Theorem:24} can be used to eliminate inappropriate sequences and has been used as a standard technique in the search for Williamson matrices and Turyn-type sequences \cite{kharaghani2005}. We can test if A, B, or C, D sequences are possibly Base sequence candidates with the fact that $(f_{A}+f_{B})(\theta)\le 4n+2$ or $(f_{C}+f_{D})(\theta)\le 4n+2$ for any $\theta$.

\section{Search algorithm for BS(n+1,n) and results}
\label{Section:search}
In this Section we describe our algorithm to search for BS(n+1,n) and the construction for BS(n+1,n) when n=41,42,43.

The main idea of our algorithms is that instead of searching for A,B,C,D sequences, we first search for A,B sequences or C,D sequences that can become part of the Base sequence. We use the properties in \Cref{Section:facts} to filter the possible candidates. Finally we fill in the remaining sequences with a backtracking procedure.

Outline of the Algorithm:

Step 1: Select (a,b,c,d) and $(a^*,b^*,c^*,d^*)$ that satisfy \Cref{Theorem:21}.

Step 2: Find all $k_{13},k_{23},k_{33},r_{13},r_{23},r_{33},p_{13},p_{23},p_{33},q_{13},q_{23},q_{33}$ satisfying \Cref{Theorem:23}.

Step 3: For every $k_{13},k_{23},k_{33},r_{13},r_{23},r_{33},p_{13},p_{23},p_{33},q_{13},q_{23},q_{33} $found in Step 2, find all $p_{16},...,p_{66},q_{16},...,q_{66}$ of which there exist at least one group of $k_{16},...,k_{66},r_{16},...,r_{66}$ so that $k_{16},...,k_{66},r_{16},...,r_{66},p_{16},...,p_{66},q_{16},...,q_{66}$ satisfies \Cref{Theorem:23}. Record all these $p_{16},...,p_{66},\\q_{16},...,q_{66}$ sets.

Step 4: For every $p_{16},...,p_{66},q_{16},...,q_{66}$ found in Step 3, generate all C,D sequences that satisfy \Cref{Theorem:22}. Then with \Cref{Theorem:24}, we eliminate the C,D sequences with $f_C(\theta)+f_D(\theta)> 4n+2$ for any $\theta=\frac{j\pi}{100} ,j=1,2...,200$.

Step 5: For each C,D sequence that passes Step 5, we search for A,B sequences by a backtracking procedure. We use the \Cref{Definition:11} and \Cref{Theorem:22} as well as the isomorphic transformation of A,B sequences to truncate branches. If the backtracking procedure ends with no solution found, we continue with the next C,D sequence until a solution is found.

With the algorithm described, we constructed BS(n+1,n) when n=41,42,43 for the first time. The result is shown in \Cref{Table:1}. 
\begin{center}
\begin{longtable}{|c|c|}
\caption{}\label{Table:1} \\
\hline
\endfirsthead
\multicolumn{2}{c}{Continued} \\
\hline
\endhead
\hline
\endfoot

BS(42,41)   &X=$+--++--++----+--++-++$\\
            &$+---++-+--++-+-+--+-+$\\
            &Y=$+++++++---++-+--++-+-$\\
            &$----++-+-+-++-+----+-$\\
            &Z=$+++-+++-----+-+++---+$\\
            &$-+++-++---+++++-++++$\\
            &W=$+---++++--++--+-+-+-+$\\
            &$--++++++++++-++-+--+$\\
\hline
BS(43,42)   &X=$++--++--+-+++-+--+-++-$\\
            &$-+----+++-+-+-+++++++$\\
            &Y=$+++++++++---+-+-+++--+$\\
            &$+-+-+-++++--++-++--+-$\\
            &Z=$+++++----+++--+++-++-+$\\
            &$--+----+--+--+++--+-$\\
            &W=$++------+---++++---+-+$\\
            &$-+++--+-++-+-++-+++-$\\
\hline
BS(44,43)   &X=$+++-+++++---+--+-++-++$\\
            &$---+++-++-+-+-+-+++--+$\\
            &Y=$+++-++--++-+-+--+--+--$\\
            &$+++---++-+-----++++---$\\
            &Z=$++---+--++++-+-+-+----$\\
            &$+--+-+-++-+++++-+-+++$\\
            &W=$--++++---++--+++++-++-$\\
            &$-+-+++++++++-++--+---$\\
\hline
\end{longtable}
\end{center}
\begin{proposition}
\label{Proposition:31}
The Base Sequence Conjecture is verified for all $n\le43$ and all Golay numbers($n=2^{a} 10^{b} 26^{c}$ with $a,b,c \ge 0$).
\end{proposition}

\section{Near-normal sequences and Yang conjecture}
\label{Section:near}
In this section and in the next section, we describe our search on near-normal and normal sequences. Given NS(n) and NNS(n) are subsets of BS(n+1,n), the algorithm in \Cref{Section:search} can be applied. Taking into account the synthesis of A,B sequences, therefore our search for NS and NNS both starts on A,B pairs. We find that all the $a,b,c,d,a^*,b^*,c^*,d^*$ that satisfy \Cref{Theorem:21} and \Cref{Corollary:21}. 
We may also consider the isomorphic transformation here. If a set of $a,b,c,d,a^*,b^*,c^*,d^*$ can be obtained from other sets by one or more of the following transformations, then we only keep one set. 

1.Negate one or more sequences of C;D

2.Reverse one or more sequences of C;D

3.Interchange the sequences C;D

4.Negate both A;B and interchange the sequences A;B

5.Alternate all four sequences of A;B;C;D

\Cref{Table:2} lists all sets of ($a,b,c,d,a^*,b^*,c^*,d^*$) that fit the above restrictions for NNS(42) and NNS(44)

\begin{center}
\begin{longtable}{|c|c|c|c|c|c|}
\caption{}\label{Table:2} \\
\hline
& $a,b,c,d$ & $a^*,b^*,c^*,d^*$ && $a,b,c,d$ & $a^*,b^*,c^*,d^*$ \\
\hline
NNS(42) & $3,-5,6,10$ & $-3,1,4,12$ & NNS(44) & $11,-7,2,2$ & $-5,9,6,6$ \\
& $3,-5,6,10$ & $-3,1,12,4$ && $13,3,0,0$ & $5,11,4,4$ \\
& $13,1,0,0$ & $3,11,2,6$ && $7,5,2,10$ & $7,5,2,10$ \\
& $11,3,2,6$ & $5,9,0,8$ && $7,5,2,10$ & $7,5,10,2$ \\
& $11,3,2,6$ & $5,9,8,0$ && $9,-9,0,4$ & $-7,7,4,8$ \\
& $9,5,0,8$ & $7,7,6,6$ && $9,-9,0,4$ & $-7,7,8,4$ \\
& $11,7,0,0$ & $9,9,2,2$ && $5,3,0,12$ & $5,3,0,12$ \\
&&&& $5,3,0,12$ & $5,3,12,0$ \\
&&&& $7,1,8,8$ & $3,5,0,12$ \\
&&&& $7,-7,4,8$ & $-5,5,8,8$ \\
&&&& $5,-5,8,8$ & $-3,3,4,12$ \\
\hline
\end{longtable}
\end{center}

We can add \Cref{Equation:14} to \Cref{Theorem:22}, then we have  
\begin{equation}
\label{Equation:42}
\left\{
\begin{array}{lll}
&r_{lm}=k_{lm}-2&,l=(n+1)\enspace mod\enspace m\\
&r_{im}=(-1)^{i-1}k_{im}&,i \ne l,1\le i\le m\\
\end{array}\right.
\end{equation}

With \Cref{Equation:42},we only need to record $k_{1m},...,k_{mm}$,then we can get $r_{1m},...,r_{mm}$. For all the sets in \Cref{Table:1}, we search for all the $k_{16},...,k_{66}$ sets with at least one group of $p_{16},...,p_{66},q_{16},...,q_{66}$ that fits \Cref{Theorem:22}.
For each $k_{16},...,k_{66}$,we generate all the sequences A,B that satisfy \Cref{Equation:26,Equation:27,Equation:14}. Then we eliminate the A,B sequences with $f_A(\theta)+f_B(\theta)> 4n+2$ for any $\theta=\frac{2j\pi}{l} ,j=1,2...,l$, we first do this for l = 50, and then for the remaining sequences we do it for l=1000. This elimination is very effective for NNS(n) and NS(n). Finally, for each remaining A,B sequence, we search for C,D sequences by a backtracking procedure. For the result obtained from our algorithm, we still need to use the isomorphic transformations described in \Cref{Section:facts} to remove duplicate sequences.

We have performed an exhaustive search with our algorithm for NNS (n) for all even $n\le44$, the results for $n\le40$ agree with the previous work in \cite{dokovic2009classication,dokovic2010new,dokovic2010hadamard}, and NNS(n) does not exist for n=42 and 44. The Yang conjecture\cite{colbourn2007}(P320) asserts that the NNS(n) is conjectured to exist for all even integers, however our search result for n=42 and 44 gives the first counter cases.

\begin{proposition}
There is no NNS(n) for n=$42$ and $44$, all other orders of even integers $n\le40$ exist. $NNS(46)$ is the first unknown case.
\end{proposition}

\section{Normal Sequences}
\label{Section:normal}
We find that all the $a,b,c,d,a^*,b^*,c^*,d^*$ that satisfy \Cref{Theorem:21} and \Cref{Corollary:21}. Similar to the Near-normal sequence, if one set can be obtained from other sets by one or more transformations listed in \Cref{Section:near}, we only keep one set. \Cref{Table:2} lists all the sets of $a,b,c,d,a^*,b^*,c^*,d^*$ for NS(n) when n=41,42,43,44,45.

\begin{center}
\begin{longtable}{|c|c|c|c|c|c|}

\caption{}\label{Table:3} \\
\hline
& ${a,b,c,d}$ & ${a^*,b^*,c^*,d^*}$ & & ${a,b,c,d}$ & ${a^*,b^*,c^*,d^*}$ \\
\hline
\endfirsthead
\multicolumn{6}{c}{Continued} \\
\hline
& ${a,b,c,d}$ & ${a^*,b^*,c^*,d^*}$ & & ${a,b,c,d}$ & ${a^*,b^*,c^*,d^*}$ \\
\hline
\endhead
\hline
\endfoot
 
NS(41) & $2,0,9,9$      & $0,2,9,9$         & NS(43)    & $-6,-8,5,-7$  & $-10,-8,3,-1$\\
       & $2,0,9,9$      & $8,10,1,1$        &           & $10,8,1,-3$   & $-10,-8,-1,3$\\
       & $2,0,9,9$      & $-4,-2,5,-11$     &           & $10,8,1,-3$   & $-10,-8,3,-1$\\
\cline{4-6}
       & $-2,-4,5,-11$  & $-4,-2,-11,5$     & NS(44)    & $7,5,10,2$    & $7,5,10,2$\\
       & $-2,-4,5,-11$  & $-4,-2,5,-11$     &           & $7,5,10,2$    & $7,5,2,10$\\
       & $-2,-4,5,-11$  & $8,10,1,1$        &           & $7,5,10,2$    & $-5,-7,10,2$\\
       & $10,8,1,1$     & $8,10,1,1$        &           & $7,5,10,2$    & $-5,-7,2,10$\\       
\cline{1-3}

NS(42) & $3,1,4,12$     & $9,7,2,6$         &           & $5,3,0,12$    & $5,3,0,12$\\
       & $3,1,4,12$     & $9,7,6,2$         &           & $5,3,0,12$    & $5,3,12,0$\\
       & $3,1,4,12$     & $-7,-9,2,6$       &           & $5,3,0,12$    & $-3,-5,0,12$\\
       & $3,1,4,12$     & $-7,-9,6,2$       &           & $5,3,0,12$    & $-3,-5,12,0$\\
\cline{4-6}
       & $3,1,4,12$     & $5,3,6,10$        & NS(45)    & $2,0,3,13$    & $0,2,3,13$\\
       & $3,1,4,12$     & $5,3,10,6$        &           & $2,0,3,13$    & $0,2,-13,-3$\\
       & $3,1,4,12$     & $-3,-5,6,10$      &           & $2,0,3,13$    & $-4,-2,-9,9$\\
       & $3,1,4,12$     & $-3,-5,10,6$      &           & $2,0,3,13$    & $4,6,3,-11$\\
\cline{1-3}

NS(43) & $2,0,1,13$     & $-2,0,-1,-13$     &           & $2,0,3,13$    & $4,6,11,-3$\\
       & $2,0,1,13$     & $-2,0,-13,-1$     &           & $2,0,3,13$    & $4,6,7,9$\\
       & $2,0,1,13$     & $-2,0,7,11$       &           & $2,0,3,13$    & $4,6,-9,-7$\\
       & $2,0,1,13$     & $-2,0,11,7$       &           & $2,0,3,13$    & $-8,-6,-1,9$\\
       & $2,0,1,13$     & $-6,-4,-1,11$     &           & $2,0,3,13$    & $-8,-6,-9,1$\\
       & $2,0,1,13$     & $-6,-4,11,-1$     &           & $2,0,3,13$    & $8,10,3,-3$\\
       & $2,0,1,13$     & $6,8,-5,7$        &           & $-2,-4,9,9$   & $-4,-2,9,9$\\
       & $2,0,1,13$     & $6,8,7,-5$        &           & $-2,-4,9,9$   & $4,6,-3,-11$\\
       & $2,0,1,13$     & $-10,-8,-1,3$     &           & $-2,-4,9,9$   & $4,6,-7,9$\\
       & $2,0,1,13$     & $-10,-8,3,-1$     &           & $-2,-4,9,9$   & $-8,-6,1,9$\\
       & $2,0,-7,-11$   & $-2,0,-7,-11$     &           & $-2,-4,9,9$   & $8,10,-3,-3$\\
       & $2,0,-7,-11$   & $-2,0,-11,-7$     &           & $6,4,3,11$    & $4,6,3,11$\\
       & $2,0,-7,-11$   & $-6,-4,-1,11$     &           & $6,4,3,11$    & $4,6,11,3$\\
       & $2,0,-7,-11$   & $-6,-4,11,-1$     &           & $6,4,3,11$    & $4,6,7,-9$\\
       & $2,0,-7,-11$   & $6,8,-5,7$        &           & $6,4,3,11$    & $4,6,-9,7$\\
       & $2,0,-7,-11$   & $6,8,7,-5$        &           & $6,4,3,11$    & $-8,-6,-1,-9$\\
       & $2,0,-7,-11$   & $-10,-8,-1,3$     &           & $6,4,3,11$    & $-8,-6,-9,-1$\\
       & $2,0,-7,-11$   & $-10,-8,3,-1$     &           & $6,4,3,11$    & $8,10,3,3$\\
       & $6,4,1,-11$    & $-6,-4,-1,11$     &           & $6,4,7,9$     & $4,6,7,9$\\
       & $6,4,1,-11$    & $-6,-4,11,-1$     &           & $6,4,7,9$     & $4,6,-9,-7$\\
       & $6,4,1,-11$    & $6,8,-5,7$        &           & $6,4,7,9$     & $-8,-6,-1,9$\\
       & $6,4,1,-11$    & $6,8,7,-5$        &           & $6,4,7,9$     & $-8,-6,-9,1$\\
       & $6,4,1,-11$    & $-10,-8,-1,3$     &           & $6,4,7,9$     & $8,10,3,-3$\\
       & $6,4,1,-11$    & $-10,-8,3,-1$     &           & $-6,-8,1,9$   & $-8,-6,1,9$ \\
       & $-6,-8,5,-7$   & $6,8,-5,7$        &           & $-6,-8,1,9$   & $-8,-6,9,1$\\
       & $-6,-8,5,-7$   & $6,8,7,-5$        &           & $-6,-8,1,9$   & $8,10,-3,-3$\\
       & $-6,-8,5,-7$   & $-10,-8,-1,-3$    &           & $10,8,3,3$    & $8,10,3,3$\\ 
\hline
\end{longtable}
\end{center}

From this step, we can prove that for $n=8k-2,k \in Z_+ $, such $a,b,c,d,a^*,b^*,c^*,d^*$ do not exist for NS(n). Normal sequences is known not exist for $n=2^{2a-1}(8b+7),a,b \in Z_+ $\cite{koukouvinos1994}, and here we gives a stronger result.

\begin{theorem}
\label{Theorem:51}
There is no normal sequence NS(n) for $n=8k-2,k \in Z_+ $.
\end{theorem}

\begin{proof} If NS(n) exist, then from \Cref{Equation:24,Equation:25,Equation:26,Equation:51}, we should find $a,b,c,d,a^*,\\b^*,c^*,d^*$ that satisfy the following equations:
\begin{equation}
\label{Equation:52}
\begin{aligned}
&a^2+b^2+c^2+d^2=4n+2\\
&a^{*2}+b^{*2}+c^{*2}+d^{*2}=4n+2 \\
&a=b+2,a^*=b^*+2\\
&a\equiv b\equiv a^*\equiv b^*\equiv 1\enspace(mod\enspace 2)\\
&c\equiv d\equiv c^*\equiv d^*\equiv 0\enspace(mod\enspace 2)\\
&a\equiv a^*+2,b\equiv b^*+2,c\equiv c^*+2,d\equiv d^*+2(mod\enspace4)\\
&c\equiv d,c^*\equiv d^*(mod\enspace4)\\
\end{aligned}
\end{equation}
We can turn the existence of $a,b,c,d,a^*,b^*,c^*,d^*$ into the following number theory problem:

Let x,y be odd integer and x=y+2, z,w be even integer, and 
\begin{equation}
\label{Equation:53}
x^2+y^2+z^2+w^2=4n+2
\end{equation}
Then for both $z\equiv w\equiv 0  (mod\enspace4)$ and $z\equiv w\equiv 2  (mod\enspace4)$ cases, such x,y,z,w exist.

However we can prove that

1)For $n=16k-10,k \in Z_+ $,such x,y,z,w do not exist for $z\equiv w\equiv 2 (mod\enspace4)$

2)For $n=16k-2,k \in Z_+ $,such x,y,z,w do not exist for both $z\equiv w\equiv 0  (mod\enspace4)$ and $z\equiv w\equiv 2  (mod\enspace4)$

For 1),we can assume x,y,z,w exists, then we can write

$x=2k_1+3,y=2k_1+1,z=4k_2+2,w=4k_3+2,k_1,k_2,k_3\in Z$

Substitute it to \Cref{Equation:53}, we have:

$4(16k-10)+2=(2k_1+3)^2+(2k_1+1)^2+(4k_2+2)^2+(4k_3+2)^2$

Reduce the equation, we have:

$8(k-1)+1=k_1(k_1+2)+2k_2(k_2+1)+2k_3(k_3+1)$

It's easy to see that $k_1$ must be odd,then $k_1(k_1+2)\equiv 3,2k_2(k_2+1)\equiv 0,2k_3(k_3+1)\equiv 0 (mod\enspace4)$

The left is equivalent to 1 mod 4 and the right is equivalent to 3 mod 4, which comes to a contradictory.

For 2),if $z\equiv w\equiv 2  (mod\enspace4)$, we write $x=2k_1+3,\ y=2k_1+1,\ z=4k_2+2,\ w=4k_3+2,\ k_1,k_2,k_3\in \mathbb{Z}$

similarly we have

$8(k-1)+5 = k_1(k_1+2) + 2k_2(k_2+1) + 2k_3(k_3+1)$

The left is equivalent to 1 mod 4 and the right is equivalent to 3 mod 4, which comes to a contradictory.

If $z\equiv w\equiv 0  (mod\enspace4)$,we write $x=2k_1+3,y=2k_1+1,z=4k_2,w=4k_3,k_1,k_2,k_3\in Z$,

Substitute it to \Cref{Equation:53}, we have:

$4(16k-2)+2=(2k_1+3)^2+(2k_1+1)^2+(4k_2)^2+(4k_3)^2$

Reduce the equation, we have:

$8(k-1)+6=k_1^2+2k_1+2k_2^2+2k_3^2$

$k_1$ must be even,we write $k_1=2r$,then we have

$4(k-1)+3=2r(r+1)+k_2^2+k_3^2$

The square of odd integer is equivalent to 1 mod 4,and the square of even integer is equivalent to 0 mod 4, and $2r(r+1)$ is equivalent to 0 mod 4, therefore the right part of the equation cannot be equivalent to 3 mod 4, which comes to a contradictory.

Finally we can see for both 1) and 2),we cannot find $a,b,c,d,a^*,b^*,c^*,d^*$ that satisfies \Cref{Equation:52}, and this proves \Cref{Theorem:51}.
\end{proof}

Similar to the Near-normal sequence, we can add \Cref{Equation:14} to \Cref{Theorem:23},then we have
\begin{equation}
\label{Equation:54}
\left\{
\begin{array}{lll}
&r_{lm}=k_{lm}-2&,l=(n+1)\enspace mod\enspace m\\
&r_{im}=k_{im}&,i \ne l,1\le i\le m\\
\end{array}\right.
\end{equation}
With \Cref{Equation:53},we only need to record $k_{1m},...,k_{mm}$,then we can get $r_{1m},...,r_{mm}$.For all the sets in \Cref{Table:1}, we search for all the $k_{16},...,k_{66}$ sets with at least one group of $p_{16},...,p_{66},q_{16},...,q_{66}$ that fits \Cref{Theorem:23}.

Here we also consider the following isomorphic transformation for NS(n)\cite{dokovic2011}:
\begin{enumerate}[leftmargin=3em]
\item Replace A with $\bar{A}=-a_1,-a_2,-a_3,-a_4,...,-a_{n-1},-a_n,a_{n+1}$

\item Replace A with $\hat{A}=a_{n},a_{n-1},a_{n-2},a_{n-3},...,a_{2},a_1,a_{n+1}$

\item Replace A with $A^*=a_1,-a_2,a_3,-a_4,...,(-1)^{n-1}a_n,a_{n+1}$
\end{enumerate}
If a set of $k_{1m},...,k_{mm}$ can be obtained from other sets by one or more transformations, then we only keep one set.

For each $k_{16},...,k_{66}$,we generate all the sequences A,B that satisfy \Cref{Equation:26,Equation:27,Equation:13}. The rest steps are exactly the same as the search for Near-normal Sequence.

We have performed exhaustive search with our algorithm to NNS(n) for all the $n\le46$, the results for $n\le40$ agree with the result in \cite{dokovic2011}, and NS(n) do not exist for n=41,42,43,44,45,46.

\begin{proposition}
\label{Proposition:52}
NS(n) exist for all the Golay numbers($n=2^{a} 10^{b} 26^{c},a,b,c \ge 0$) and $n=3,5,7,9,11,12,13,25,29$. NS(n) do not exist for $n=8k-2,k \in Z_+ $ and all the other $n\le46$ , the first unknown case is $NS(47)$.
\end{proposition}
\section{Acknowledgements}
The work was supported in part by The Startup Foundation for Introducing Talent
of NUIST. We acknowledge the High Performance Computing Center of Nanjing University of Information Science and Technology for their support of this work.









\bibliographystyle{elsarticle-num}
\bibliography{cas-refs}



\end{document}